\documentclass[12pt,reqno]{amsart}

\usepackage{ljm-auth}
\usepackage{amsfonts,amssymb,mathrsfs,amscd}
\usepackage{enumerate}
\usepackage[pdftex,unicode,colorlinks,linkcolor=blue,
citecolor=red,bookmarksopen,pdfhighlight=/N]{hyperref}

\numberwithin{equation}{section}
\theoremstyle{plain}
\newtheorem{theorem}{Theorem}
\newtheorem{theoA}{\bf Theorem A}

\newtheorem{propos}{\bf Proposition}

\theoremstyle{definition}
\newtheorem{definition}{Definition}

\newtheorem{remark}{Remark}
\newtheorem{example}{Example}

\newcommand{\e}{\varepsilon}
\newcommand{\leftd}{\text{\tiny\rm left}}
\newcommand{\rightd}{\text{\tiny\rm right}}
\renewcommand{\d}{\,{\rm d}}
\def\RR{\mathbb R}
\def\CC{\mathbb C}
\def\NN{\mathbb N}

\DeclareMathOperator{\spr}{sp}
\DeclareMathOperator{\clos}{clos}
\DeclareMathOperator{\Int}{int}

\DeclareMathOperator{\Hol}{Hol}
\DeclareMathOperator{\Zero}{Zero}

\DeclareMathOperator{\sbh}{sbh}
\DeclareMathOperator{\im}{im}
\DeclareMathOperator{\supp}{supp}

\author{B.\,N.\,Khabibullin, N.\,R.\,Tamindarova}
\crauthor{B.\,N.\,Khabibullin, N.\,R.\,Tamindarova} % for running heads

\tit{Subharmonic  test functions\\ and the distribution of zero sets\\ of holomorphic functions}
\shorttit{Subharmonic  test functions\dots} %for running heads

\begin{document}

\maketit

\abstract{Let $m,n\geq 1$ are integers and $D$ be a domain in the complex plane $\CC$ or in the $m$-dimensional real space $\RR^m$. We build positive subharmonic functions on $D$ vanishing on the boundary $\partial D$ of domain $D$.
	 We use such (test) functions to study the distribution of zero sets of holomorphic functions $f$ on $D\subset \CC^n$  with restrictions on the growth of $f$ near the boundary $\partial D$.} \notes{0}
{\subclass{31A05, 31B05, 32A60, 30F45}
\thank{ Partially supported by Russian Foundation for Basic Research, Grant 16--01--00024}}

% %Introduction
% %1.1
% %Problem Statement and Motivation
% %Introduction and Motivation
\section{Motivation and statement of the problem}\label{Spsm}
\subsection{Notations, definitions, and agreements}\label{nda} We use an information and definitions from \cite{HK}--\cite{H}.
As usual, $\mathbb N:=\{1,2, \dots\}$, $\mathbb R$ and $\mathbb C$ are the sets of all natural, real and complex numbers, resp. 
We set  
\begin{subequations}\label{df:R}
	\begin{align}
	\RR_{-\infty}:=\{-\infty\}\cup \RR,\; 	\RR_{+\infty}&:=\RR\cup \{+\infty\}, 
	\, \RR_{\pm\infty}:=\RR_{-\infty}\cup \RR_{+\infty},
	\tag{\ref{df:R}r}\label{df:Rr}\\
	\RR^+:= \{x\in \RR\colon x\geq 0\},\; \RR_*^+&:=\RR^+\setminus \{0\}, \; \RR_{+\infty}^+:= \RR^+\cup \{+\infty\},
	\tag{\ref{df:R}p}\label{df:R+}
	\end{align}
 \end{subequations} 
where the usual  order relation $\leq$ on $\RR$  is complemented by the inequalities
$-\infty \leq x\leq +\infty$ for all $x\in \RR_{\pm\infty}$. 
Let $f\colon X\to Y$ be a function. For $Y\subset \RR_{\pm\infty}$, $g\colon X_1\to \RR_{\pm\infty}$ and $S\subset X\cup X_1$, we write ``$f = g$ {\it on\/} $S$\,'' or  ``$f \leq g$ {\it on\/} $S$\,'' if $f(x)= g(x)$ or $f(x)\leq g(x)$ for all $x\in S$ respectively.

Let $m\in \NN$. Denote by $\mathbb R^m$ the {\it  $m$-dimensional Euclidian real  space.\/} Then $\mathbb R^m_{\infty}
:=\mathbb R^m \cup \{\infty\}$ is the {\it  Alexandroff\/} ($\Leftrightarrow$one-point) {\it compactification of\/} $\mathbb R^m$.
Given a subset $S$ of $\mathbb R^m$ (or $\mathbb R^m_{\infty}$), the closure $\clos S$, the interior $\Int S$  and the boundary $\partial S$ will always be taken relative $\mathbb R^m_{\infty}$.

 Let $S_0\subset S\subset \mathbb R^m_{\infty}$. If  the  closure $\clos S_0$ 
is a compact subset of $S$ in the topology induced on $S$  from $\mathbb R^m_{\infty}$, then
we write $S_0\Subset S$. An open connected (sub-)set of $\RR^m$ is a {\it (sub-)domain}  of $\RR^m_{\infty}$.
Given $x\in \RR^m$ and $r\overset{\eqref{df:R}}{\in} \RR_{+\infty}^+$, we set 
$B(x,r):=\{x'\in \RR^m \colon |x'-x|<r\}$,
where $|\cdot |$ is the Euclidean norm  on  $ \RR^m$, $|\infty|:=+\infty$; $B(r):=B(0,r)$. 
Besides, $B(\infty,r):=\{x\in \RR^m \colon |x|>1/r\}$, $\overline{B}(x,r):=\clos B(x,r)$
for $r>0$, but $\overline{B}(x,0):=\{x\}$ and $\overline{B}(+\infty):=\RR^m_{\infty}$.  

Let $A, B$ are sets, and $A\subset B$. The set $A$   is a {\it non-trivial subset\/} of the set $B$ if the subset $A\subset B$ is non-empty ($A\neq \varnothing$) and   {\it proper\/} ($A\neq B$).

We understand always the {\it ``positivity''\/} or {\it ``positive''\/} as $\geq 0$, where  the symbol $0$ denotes the number zero, the zero function, the zero measure, etc. So, a function 	$	f\colon X\to R\overset{\eqref{df:R}}{\subset} \RR_{\pm\infty}$ 
is positive on $X$ if 	$f(x)\geq 0$ for all $x\in X$. In such case we write ``$f\geq 0$ {\it on\/} $X$''.

The operation of superposition of functions denoted by $\circ$.

By $\mathcal M^+(S)$ denote the class of all Borel positive measures on $S$. 

Let $\mathcal O$ be a non-trivial  open subset of $\mathbb R^m _{\infty}$.
 We denote by $\sbh (\mathcal O)$ the class of  all subharmonic functions $u\colon \mathcal O\to \RR_{-\infty}$ on $\mathcal O$ for $m\geq 2$, and all (local) convex functions  $u\colon \mathcal O\to \RR_{-\infty}$ on $\mathcal O$ for $m=1$. 
 By $\star\colon \RR_{\infty}^m\to \RR_{\infty}^m$ denote the  inversion in the unit  sphere $\partial B(0,1)$:
 \begin{equation}\label{stK}
 \star \colon x\mapsto x^{\star}:= \begin{cases}
 0\quad&\text{for $x=\infty$},\\
 \frac{1}{|x|^2}\,x\quad&\text{for $x\neq 0,\infty$},\\
 \infty\quad&\text{for $x=0$}.
 \end{cases}
 \end{equation}
 A function $u$ is subharmonic on a neighbourhood of $\infty \in R_{\infty}^m$ if its {\it Kelvin transform\/}
 \begin{equation}\label{stK+}
 u^{\star}(x^{\star})=|x|^{m-2}u(x), \quad x ^{\star}\in \mathcal O^{\star}:=\{x^{\star}\colon x\in
  \mathcal O\},
 \end{equation}
  is subharmonic on a neighbourhood of $0$.
  The class $\sbh (\mathcal O)$  contains the function 
 $\boldsymbol{-\infty}\colon x\mapsto -\infty$, $x\in \mathcal O$ (identically equal to $-\infty$); 
 \begin{equation*}
\sbh^+(\mathcal O):=\{u\in \sbh (\mathcal O)\colon u\geq 0 \text{ on $\mathcal O$}\},\quad  \sbh_*(\mathcal O):=\sbh\,(\mathcal O)\setminus \{\boldsymbol{-\infty}\}.
 \end{equation*}
  For $u\in \sbh_*(\mathcal O)$, the {\it  Riesz measure of\/} $u$ is the  Borel  (or Radon \cite[A.3]{R})  positive measure 
\begin{equation}\label{df:cm}
\nu_u:= c_m \,\Delta u\in \mathcal M^+(\mathcal O),  \quad c_m:=\frac{\Gamma(m/2)}{2\pi^{m/2}\max\bigl\{1, (m-2)\bigr\}}\,,
\end{equation}
where $\Delta$ is  the {\it Laplace operator\/}  acting in the sense of distribution theory, and $\Gamma$ is the gamma function. In particular,  $\nu_u(S)<+\infty$ for each subset $S\Subset \mathcal O$.
By definition, $\nu_{\boldsymbol{-\infty}}(S):=+\infty$ for all $S\subset \mathcal O$. 

\subsection{Test functions} Subjects of our investigation are presented by
\begin{definition}\label{dftf} Throughout what follows 
	$m,n\in \NN$ and  $ \varnothing \neq K=\clos K\Subset D\subset \RR_{\infty}^m$, where 
	 	$D$ is {\it a subdomain in\/} $\mathbb R^m_{\infty}$ or  $\CC^n_{\infty}$.
	A function $v\in \sbh^+ (D\setminus K)$ is % %called 
	a {\it  test function for\/ $D$ outside of\/ $K$} if 
\begin{subequations}\label{v0l}
\begin{align}
\lim_{D\ni x'\to x} v(x')&=0 \quad \text{for each  $x\in \partial D$}
\tag{\ref{v0l}o}\label{v0lo}
\\
\text{and}\quad \sup_{x\in D\setminus K}v(x)&<+\infty.
\tag{\ref{v0l}s}\label{v0ls}
\end{align}
\end{subequations}
	The class of  test functions for $D$ outside of $K$ is denoted by $\sbh_0^+(D\setminus K)$.
\end{definition}

We give {\it elementary properties\/} and simple examples of test functions.

\begin{enumerate}[{\bf t}1.]\label{pro:tf}
\item\label{pro:tf1} The condition \eqref{v0lo}   can be replaced by the condition:
{\it for each number  $\e\in \RR^+_*$ there is a subset $S_{\e}\Subset D$ such that $0\leq v<\e$ on $D\setminus S_{\e}$.}  
\item \label{pro:tf2} If a function $v\in \sbh_0^+(D\setminus K)$  is continued (extended) by zero  as
\begin{equation}\label{d:exvv}
v(x):=\begin{cases}
v(x), \quad &\text{for  $x\in D\setminus K$,}\\
0, \quad & \text{for  $x\in \RR_{\infty}^m\setminus D$,}		
\end{cases}
\end{equation}
then the extended function $v$ is a subharmonic function on $\RR_{\infty}^m\setminus K$ and 
$v\in \sbh_0^+(\RR_{\infty}^m\setminus K)$. 
\item\label{pro:tf3} % %Under the condition  \eqref{dfK}, 
If $v\in \sbh_0^+(\RR_{\infty}^m\setminus K)$ and $v=0$ on $\RR_{\infty}^m\setminus D$, then  $v\in \sbh_0^+(D\setminus K)$. % %Further 
Throughout what follows we identify a test function $v\in \sbh_0^+(D\setminus K)$ and its continuation \eqref{d:exvv} of  the class 
$\sbh_0^+(\RR_{\infty}^m\setminus K)$.
\item\label{pro:tf4} The condition \eqref{v0ls} can be replaced by the condition
\begin{equation*} % %\label{lbvS}
\sup_{x\in \partial  S}\limsup_{D\setminus {S} \ni x'\to x} v(x')<+\infty \;\text{(the maximum Principle for
	$\sbh (\RR_{\infty}^m\setminus K)$)}.
\end{equation*}  
\item\label{pro:tf5} If  $v\in \sbh_0^+(D\setminus K)\subset \sbh_0^+(\RR_{\infty}^m\setminus K)$, then its Riesz measure $\nu_v$ belongs to $\mathcal M^+(\clos D\setminus K)$. % % and  is finite,  i.\,e. 
% %\begin{equation}
% %$\nu_v(\clos D\setminus K)<+\infty$. 
% %\end{equation}
\end{enumerate}

\begin{example}\label{exg} Let $D\subset \RR^m_{\infty}$ be a domain, $\widetilde D\subset D$ a regular (for the Dirichlet problem) subdomain  of $D$, and $\exists\, x_0\in \widetilde D$. Then {\it the extended Green's function\/ $g_{\widetilde D}(\cdot, x_0)$ for $\widetilde D$ with pole at\/} $x_0$ is a test function 
from the class $\sbh_0^+\bigl(D\setminus \{x_0\}\bigr)$. Its Riesz measure  is {\it the harmonic	measure\/
$\omega_{\widetilde D}(x_0,\cdot)$ for\/ $\widetilde D$ at\/} $x_0$ such that 
\begin{equation*}
 \supp \omega_{\widetilde D}(x_0,\cdot)\in \mathcal M^+(\partial \widetilde D)\subset \mathcal M^+(\clos D), \quad
 % %\omega_D(x_0,\clos D)=
 \omega_{\widetilde D}(x_0,\partial \widetilde D)=1.
 \end{equation*}
 \end{example}

\subsection{Holomorphic functions} Let $n\in \NN$. Denote by $\mathbb C^n $ the {\it  $n$-dim\-e\-n\-s\-i\-o\-n\-al Euclidian complex  space.\/} Then $\mathbb C^n _{\infty}:=\mathbb C^n \cup \{\infty\}$ is the {\it  Alexandroff\/} ($\Leftrightarrow$one-point) {\it compactification of\/} $\mathbb C^n$.  If it is necessary,  
we identify $\mathbb C^n $ (or $\mathbb C^n _{\infty}$) with $\mathbb R^{2n}$ (or $\mathbb R^{2n}_{\infty}$). Let $\mathcal O$ be a non-trivial  open subset of $\mathbb C^n_{\infty}$. We denote by $\Hol (\mathcal O)$ and $\sbh (\mathcal O)$ the class of holomorphic  and subharmonic functions on $\mathcal O$, resp. For $u\in \sbh_*(\mathcal O)$, the {\it  Riesz measure of\/} $u$ is the  Borel (and  the Radon) positive measure 
\begin{equation}\label{df:cmf}
\nu_u\overset{\eqref{df:cm}}{:=}c_{2n} \, \Delta u\in \mathcal M^+(\mathcal O), \quad c_{2n}\overset{\eqref{df:cm}}{:=}\frac{(n-1)!}{2\pi^n \max\{1,2n-2\}}\,.
\end{equation}
For $k\in \{0\}\cup \NN$, we denote by $\sigma_k$ the {\it $k$-dimensional surface\/} ($\Leftrightarrow$Hausdorff) {\it measure\/} on $\mathbb C^n$
and its restrictions to subsets of $\mathbb C^n$. So, if $k=0$, then $\sigma_0(S)=\sum_{z\in S}1$ for each $S\subset \mathbb C^n$, i.\,e. $\sigma_0(S)$ is equal to the number of points in the set $S\subset \mathbb C^n$.

\begin{theoA}[{see \cite[Corollary 1.1]{KhT15} for the case $n=2$, and 
\cite{KhT16_b}--\cite[Corollary 1]{KhT16_c} for $n>1$}]\label{th:s} 
	Let\/ $D$ be a non-trivial domain in $\CC^n_{\infty}$, $K$ a  compact subset of $D$
	with  $\Int K\neq \varnothing$. Let  $M\in \sbh_*(D)$ be a function  with the Riesz measure $\nu_M\in \mathcal M^+(D)$, and $ v\in \sbh_0^+(D\setminus K)$ a test function for $D$ outside of $K$. Assume that
	\begin{equation}\label{co:M}
	\int_{D\setminus K} v\, {\rm d} \nu_M<+\infty.
	\end{equation}
	Let  $f\in \Hol (D)$ and 	$\Zero_f:=\{z\in D \colon f(z)=0\}\supset {\tt Z}$. 
	If 
	\begin{equation}\label{h:}
	 		|f|\leq e^M\text{ on } D,
\quad 	\int_{{\tt Z}\setminus K}v\, {\rm d} \sigma_{2n-2}=+\infty,
\end{equation}
 then $f= 0$ on $D$, i.\,e. $\Zero_f=D$.
\end{theoA}
This Theorem A shows that each constructed test function of the class $\sbh_0^+(D\setminus K)$ gives a uniqueness theorem in terms of the distribution of the zero set of holomorphic functions. The main goal of our article is to give some methods for constructing of test functions  in the sense of Definition \ref{dftf}
with applications to the distribution of the zero sets of holomorphic functions. 
Many such constructions have been proposed for domains in the complex plane $\CC$ 
in our work \cite[sections 4--5]{KhT15}. 

\section{Radial case}
\subsection{Radial subharmonic  functions}  
A subset $S\subset R_{\infty}^m$ is {\it radial,\/} if from the conditions $x\in S$ and $|x'|=|x|$ it follws that $x'\in S$. 
 A function $f$ on radial set  $S$ is {\it radial,\/} if  $f(x)=f(x')$ for all $|x|=|x'|$, $x\in S$. 
By  $\im f$ denote the image of $f$. Further
 \begin{equation*}
\spr S:=\{|x|\colon x\in S\},\quad 
 \spr_f\colon \spr S\to \im f, \quad \spr_f(r) :=f\bigl(|x|\bigr) \text{ for $r=|x|$,}
 \end{equation*} 
 is {\it the spherical projection\/}  of radial function $f$ on radial set $S$.
 Let $0\leq r_1<r_2\leq +\infty$ and $h\colon (r_1,r_2)\to \RR$ be a strictly increasing function.   A function 
  $f\colon (r_1,r_2)\to \RR$ is {\it convex of\/} $h$   if the function $f\circ h^{-1}$ is convex on $\bigl(h(r_1),h(r_2)\bigr)\subset \RR$.   Given $t\in \RR_*^+$, we set 
 \begin{equation*}
 h_m(t):=
 \begin{cases}
 t\quad &\text{for $m=1$},\\
 \log t\quad &\text{for $m=2$},\\
 -\dfrac{1}{t^{m-2}}\quad &\text{for $m\geq 3$},
 \end{cases}
 \qquad t\in \RR_*^+;
 \end{equation*}
  \begin{equation}\label{Arr}
 A(r_1,r_2):=\{ x\in \RR^m\colon r_1<|x|<r_2 \}.
 \end{equation}

\begin{propos}\label{pr:hc}  Let $Q\colon A(r_1,r_2)\to \RR$ be a radial function and
	$q:=\spr_Q$.  The following five conditions are equivalent:
\begin{enumerate}[\rm I.]
	\item\label{1s}   The function  $Q$ is subharmonic on  $A(r_1,r_2)$, $Q\neq \boldsymbol{-\infty}$. 	
	\item\label{2s} The function  $q$ is convex of $h_m$ on $\bigl(h_m(r_1),h_m(r_2)\bigr)\subset \RR$.
	\item\label{3s} The function $q$ has the following properties: {\rm i)}
$q$ is continuous, {\rm ii)} there exist the left derivative $q'_{\leftd}$ and the right derivative
$q'_{\rightd}(r)$, {\rm iii)} $q'_{\leftd}$ is  continuous on the left, and $q'_{\rightd}$ is  continuous on the right, {\rm iv)}  the functions  $r\mapsto r^{m-1}q'_{\leftd}(r)$,  $r\mapsto r^{m-1}q'_{\rightd}(r)$ are increasing, {\rm v)} $ q'_{\leftd}\leq q'_{\rightd}$  on $(r_1,r_2)$ 
{\rm vi)} there is a no-more-than countable set $R\subset (r_1,r_2)$ such that 
$ q'_{\leftd}= q'_{\rightd}$ on $(r_1,r_2)\setminus R$.
	\item\label{4s} For any $r_0\in (r_1,r_2)$ there is an increasing function $p_0\colon (r_1,r_2)\to \RR$ such that 
	\begin{equation*}
	q(r)=q(r_0)+\int_{r_0}^r \frac{p_0(t)}{t^{m-1}} \d t, \quad  r\in (r_1,r_2), 
	\end{equation*}
where  the function $p_0$  can be chosen in the form
	\begin{equation*}
	p_0(r):=	r^{m-1}q'_{\leftd}(r) \quad \text{ or } \quad  p_0(r):= r^{m-1}q'_{\rightd}(r), \quad   r\in (r_1,r_2).
	\end{equation*}
	\item\label{5s} The function  $q$ is upper semicontinuous, locally integrable on $(r_1,r_2)$, and  
	$r\mapsto \bigl(r^{m-1}q'(r)\bigr)'$ is a positive distribution (measure).
\end{enumerate}
\end{propos}
The proof is omitted (see \cite{HK}--\cite{H} and \cite[\S~4]{KhT15} for $m=2$ or $\CC$). 

\subsection{Radial test functions}\label{ssCrC} Let $0<r_0<R\in \RR_{+\infty}^+$.  The following statement describes all radial test functions for  the domain $D=B(R)\subset \RR^m_{\infty}$. Recall that 
$B(+\infty)=\RR^m$.
\begin{propos}\label{prsds} Let $v\colon B(R)\setminus \overline B(r_0) \to \RR^+$ 
	be a radial function on    
	 $B(R)\setminus \overline B(r_0)$. The following three conditions are equivalent:
	\begin{enumerate}[\rm 1.]
		\item\label{r_1} The function $v$ is a test function  for\/ $B(R)$ outside of\/ $\overline B(r_0)$.
		\item\label{r_2} There is a decreasing function $d\colon (r_0,R)\to \RR^+$ such that
		\begin{equation*}
		v(x)=\int_{|x|}^{R} \frac{d(t)}{t^{m-1}} \d t<+\infty, \quad x\in B(R)\setminus \overline B(r_0).
			\end{equation*}
		\item\label{r_3} The function $\spr_v\circ h_m^{-1}$ is convex on $\bigl(h_m(r_0),h(R)\bigr)$ and $$\lim\limits_{h_m(R)>x\to h_m(R)} \spr_v (x)=0.$$
	\end{enumerate}
	\end{propos}
	 \begin{proof} If we apply the inversion and the Kelvin transform from \eqref{stK}--\eqref{stK+}
	 	 to the extended  function \eqref{d:exvv} with $D=B(R)$ and $K=B(r_0)$, then 
the equivalences \ref{r_1}$\Leftrightarrow$\ref{r_2}  and \ref{r_1}$\Leftrightarrow$\ref{r_3}   follow from the equivalences \ref{1s}$\Leftrightarrow$\ref{4s} and \ref{1s}$\Leftrightarrow$\ref{2s} of Proposition \ref{pr:hc} respectively.
\end{proof}
We can easily add other equivalences  to Proposition \ref{prsds} based on Proposition \ref{pr:hc}.

\subsection{Cases $D=B(R)$}\label{ssCrC+}
The following result follows immediately from Theorem A and Proposition \ref{pr:hc} (see \cite[\S~4]{KhT15} for $m=2$ or $\CC$). 

\begin{theorem} Let $0<r_0<R\in \RR_{\infty}^+$. Let $M\colon B(R)\to \RR$ be a continuous 
	radial function and $q:=\spr_M$. Suppose that $q$ is convex of $h_{2n}$
	on $(0,R)$ and  there is a decreasing function $d\colon (r_0,R)\to \RR^+$ such that
	\begin{equation*}
	\int_{r_0}^R d(r)q_{\rightd}'(r) \d r\overset{\eqref{co:M}}{<}+\infty. 
	\end{equation*}
	If the function $f\in \Hol \bigl(B(R)\bigr)$ with zero set $\Zero_f\supset  {\tt Z} $ satisfies the conditions
	{\rm (see \eqref{h:})}
\begin{equation*}\label{h:1}
|f|\leq e^M\text{ on } B(R),
\quad 	\int_{r_0}^R d(r)s_{\tt Z}(r) \frac{\d r}{r^{2n-1}}=+\infty,
\end{equation*}
where $s_{\tt Z}(r)=\sigma_{2n-2}\bigl(Z\cap B(r)\bigr)$, $r\in (r_0,R)$, 
then $f= 0$ on $D$. 
	\end{theorem}	
The proof is a direct computation of \eqref{co:M}--\eqref{h:}  for radial case 
with $D=B(R)$ and $K=\overline B(r_0)$ 
using the integration by parts.	So, in \eqref{co:M},
\begin{equation*}
\d \nu_M (rz)\overset{\eqref{df:cmf}}{=}c_{2n}\d  \bigl(r^{m-1}q_{\rightd}'(r)\bigr)\otimes \d \sigma_{2n-1}(z), \quad z\in \partial 
B(1),
\end{equation*}
and we consider the test function $v$ from Proposition \ref{prsds}\eqref{r_2}, in \eqref{h:}, 
\begin{equation*}
\int_{{\tt Z}\setminus B(r_0)}v\, {\rm d} \sigma_{2n-2}
=\int_{r_0}^R d(r)s_{\tt Z}(r) \frac{\d r}{r^{2n-1}} -s_{\tt Z}(r_0) \int_{r_0}^R \frac{d(t)}{t^{m-1}} \d t.
\end{equation*}
Radial test functions can also be considered for sets $A(r_1,r_2)\overset{\eqref{Arr}}{\subset} \CC^n$.
But it is not of interest for holomorphic functions on $A(r_1,r_2)\subset \CC^n$ in view of the 
 classical Hartogs extension phenomenon. Here we do not consider also holomorphic functions on polydiscs  in $\CC^n$, $n>1$.
 
 \section{Green's case}\label{Gc}
 Throughout this section \ref{Gc} $D\subset \RR_{\infty}^m$ is a regular domain 
  with Green's function $g_D:=g_D(\cdot, x_0)$ (with the pole at $x_0\in D$). 
  We set 
\begin{equation}\label{Dt}
D_t:=\{x\in D\colon g_D(x,x_0)>t\}\ni x_0 , \quad 0<t\leq t_0\in \RR_*^+.
\end{equation}

 \subsection{Superpositions of  convex functions and Green's functions} 
 \begin{propos}\label{pr_rf}  Let  $q\colon [0,t_0)\to \RR^+$ be a convex function such that $q(0)=0$. Then the superposition 
 	 $q\circ g_D$ is  a  test function for $D$ outside of $D_{t_0}$, i.\,e.  $q\circ g_D\in \sbh_0^+(D\setminus \overline D_{t_0})$. 
  \end{propos}
 \begin{proof} The superposition of convex function $f$ and harmonic function $g_D(\cdot,x_0)$ 
 	is subharmonic.  For $v:=f\circ g_D$, the condition \eqref{v0lo}  follows from the condition  $f(0)=0$, since 
 	the Green's function $g_D(\cdot , x_0)$ vanishes on the boundary $\partial D$ of regular domain $D$. 
 	\end{proof}
 	
 	\begin{propos}\label{pr_rF} Let $F\colon (-\RR^{+}_*)\to \RR$ be a  convex increasing function, $F(-\infty):=\lim\limits_{x\to -\infty} F(x)\in \RR_{-\infty}$, where $(-\RR^+_*):=\RR_{-\infty}\setminus \RR^+$. Then the  superposition  $F\circ (- g_D)$ is  subharmonic on $D$. 
 	\end{propos}
 	\begin{proof} Obviously, the function 
 		$-g_D(\cdot, x_0)$ is subharmonic on $D$.  The superposition of convex increasing function $F$ and subharmonic function $-g_D(\cdot,x_0)$ is subharmonic on $D$.  
 	\end{proof}
 	
 	\subsection{A uniqueness theorem with Green's functions} 
 		
 For simplicity, we assume that the boundaries  $\partial D$ and $\partial D_t$ of $D_t$ from \eqref{Dt}
   belong to   the class	$C^2$.
   \begin{theorem}[{see \cite[Theorem  7]{KhT15} for $n=1$}] 
   	Suppose that the functions $q$ and $F$ are the same as in Propositions\/ {\rm \ref{pr_rf}} and\/ 
   	{\rm \ref{pr_rF}.} Let $q\in C^1(0,t_0)$ and $F\in C^1(-\RR^{+}_*)$, and
 	\begin{equation}\label{qFc}
\int_0^{t_0} q'(t)F'(-t) \d t <+\infty.
\end{equation}
   	 If the function $f\in \Hol (D)$ with zero set $\Zero_f\supset  {\tt Z} $ satisfies the conditions
     	 \begin{equation}\label{h:2}
   	 |f|\overset{\eqref{h:}}{\leq} \exp \bigl(F\circ (-g_D)\bigr)\text{ on } D,
   	 \quad 	\int_{0}^{t_0} q'(t)s_{{\tt Z},D}(t) \d t=+\infty,
   	 \end{equation}
   	 where $s_{{\tt Z},D}(t)=\sigma_{2n-2}\bigl({\tt Z}\cap D_t\bigr)$, $t\in (0,t_0)$, 
   	 then $f= 0$ on $D$. 
    \end{theorem}
  \begin{proof}Let $\nu_M$ be the Riesz measure of   $M:=F\circ (-g_D)\in \sbh(D)$, and 
  	$v:=q\circ g_D$. We have the following equalities:
  	\begin{subequations}\label{meq}
\begin{align}
\nu_M (D_{t_2}\setminus \overline D_{t_1})=F'(-t_2)-F'(-t_1), \quad -t_0<-t_1<-t_2<0 &\text{ \cite[6.2.1]{KhT15}},
\notag
\\
\int_{D\setminus \overline D_{t_0}} v\d \nu_M =\int_{D\setminus \overline  D_{t_0}} (q\circ g_D)\d \nu_M =\int_0^{t_0} q(t)\d \bigl(-F'(-t)\bigr),&
\tag{\ref{meq}M}\label{meqM}
\\
	\int_{{\tt Z}\setminus \overline D(t_0)}v\, {\rm d} \sigma_{2n-2}=
	\int_{{\tt Z}\setminus \overline D(t_0)}(q\circ g_D) \, {\rm d} \sigma_{2n-2}=
	\int_{0}^{t_0} q(t) \d s_{{\tt Z},D}(t). &
	\tag{\ref{meq}Z}\label{meqz}
	\end{align}	
  	\end{subequations}	
Next we apply the integration by parts to the right-hand sides of \eqref{meqM}--\eqref{meqz}
 and Theorem A with $K=\overline D_{t_0}$.
  	\end{proof} 
  	\begin{remark}
  	The conditions to the boundaries $\partial D$ and $\partial D_t$ can be considerably weakened \cite{Widman}.
  	In addition, if we replace the derivatives $q',F'$  by $q'_{\rightd}, F'_{\rightd}$ in \eqref{qFc} and by $q'_{\leftd}$
  	in \eqref{h:2} respectively,  we can remove the conditions $q\in C^1(0,t_0)$ and $F\in C^1(-\RR^{+}_*)$.
  	\end{remark}  

We will provide a more general and subtle results on the test functions and their construction elsewhere.

\noindent
 Bulat Nurmievich Khabibullin, Nargiza Rustamovna Tamindarova; \\
 Bashkir State University; E-mail: Khabib-Bulat@mail.ru

\end{document}